\documentclass{amsart}
\usepackage[all]{xy}
\usepackage{amssymb,latexsym, amscd, amsthm, amsmath, upgreek, bm, dsfont, mathrsfs}
\newtheorem{theorem}{Theorem}[section]

\theoremstyle{definition}
\newtheorem{definition}[theorem]{Definition}

\theoremstyle{remark}
\newtheorem{remark}[theorem]{\bf Remark}

\newcommand{\C}{\mathbb C}

\newcommand{\G}{\mathbb G}
\newcommand{\Z}{\mathbb Z}
\newcommand{\R}{\mathbb R}

\newcommand{\SSS}{\mathbb S}

\begin{document}

\title[Hodge structures and Weierstrass $\sigma$-function]{Hodge structures and Weierstrass $\sigma$-function}

\author[G. Banaszak]{Grzegorz Banaszak*}
\address{Department of Mathematics and Computer Science, Adam Mickiewicz University,
Pozna\'{n} 61614, Poland}
\email{banaszak@amu.edu.pl}

\author[J. Milewski]{Jan Milewski**}
\address{Institute of Mathematics, Pozna\'n University of Technology,
ul. Piotrowo 3A, 60-965 Pozna\'n, Poland}
\email{jsmilew@wp.pl}

\date{}

\thanks{*Partially supported by the NCN (National center of Science for Poland)
NN201 607440 \newline
**Partially supported by the NCN grant NN201 373236}

\begin{abstract}
{In this paper we introduce new definition of Hodge structures and show that $\R$-Hodge structures are 
determined by $\R$-linear operators that are annihilated by the Weierstrass $\sigma$-function}
\end{abstract}

\maketitle


\section{Introduction}
 Classically a Hodge structure of a given weight can be defined in the four equivalent ways as follows (see e.g.
 \cite{G}, \cite{PS}):

\begin{definition}  {A Hodge structure of a weight $n$ on a real vector space $V$} consists of a 
finite--dimensional $\R$--vector space $V$ together with any of the following equivalent data:

(i) A decomposition $V_{\C} =\oplus _{p+q=n}V^{p,q}$, called the 
{\it Hodge decomposition}, such that 
$\overline{V^{p,q}}=V_{q,p}$.

(ii) A decreasing filtration $F^r_H V_{\C}$ of $V_{\C}$, called the 
{\it Hodge filtration}, such that 
$F^r_H V_{\C}\oplus \overline{V^{n-r+1}_{\C}}=V_{\C}$.

(iii) A homomorphism $h_n: \SSS \rightarrow {\rm GL} (V_{\R})$ of real algebraic groups, and also 
specifying that the weight of the Hodge structure is $n$.

(iv) A homomorphism $h_n: \SSS  \rightarrow {\rm GL} (V_{\R})$ of real algebraic groups such that via the decomposition 
$\G_{m}/{\R} \rightarrow \SSS  \rightarrow {\rm GL} (V_{\R})$ an element $t \in \G_{m}/{\R}$ acts as 
$t^{-n}\cdot Id$.
\label{def of Hodge str}\end{definition}

Throughout the paper we work with Hodge structures of various weights, hence by a Hodge structure we 
understand here
a finite direct sum 
\begin{equation}
\rho \,\,\, :=\,\,\,  \bigoplus_{j=1}^{k} \,\, h_{n_j}
\label{general Hodge structure}
\end{equation} 
of representations $h_{n_j}$ described in (iii) or (iv) of the Definition \ref{def of Hodge str}.

In this paper we consider Hodge structures on real vector space $V$ via representations of the Lie algebra of the real 
algebraic group $\SSS$ (denoted also ${\C}^{\times}$) on $V$. In section 2
we show that a Hodge structure can be treated as a pair of operators $E, T$ na $V$ satisfying certain conditions
(see Theorem  \ref{Hodge structure via E and T }).
In section 3 we show that a Hodge structure can be treated as a single operator $S := E + T$ on $V$ such that $\sigma (S) = 0$
for a Weierstrass $\sigma$-function
which corresponds to decomposition of $V$ into eigenspaces of operators $E$ and $T.$ 
Weirstrass  $\sigma$-function does not have multiple zeros hence this corresponds to the fact that complexification of 
$S$ does not have generalized eigenvectors other than ordinary ones. 

\section{Hodge structures and Lie algebras.}
\medskip

\noindent
The following theorem gives another definition of the Hodge structure.

\noindent
\begin{theorem} \label{Hodge structure via E and T } Let $V$ be a finite dimensional vector space over 
$\R$. 
There is a one to one correspondence between the family of Hogde structures on $V$ and the family of pairs of endomorphisms 
$E,T\in {\rm End}_{\R} (V)$ satisfying the following conditions:
\begin{equation} \label{hp}  
[E,T]=0, \quad \sin (\pi E) =0, \quad \sinh (\pi T) =0 ,
\end{equation}  
\begin{equation} \label{hk} 
\sin (\frac{\pi}{2} (E^2+T^2))=0 
\end{equation}
\end{theorem}

\begin{proof} Consider a Hodge structure on $V$. By (\ref{general Hodge structure}) 
(cf. Definition \ref{def of Hodge str} (iii))  this gives a representation:
$$\rho\, :\, \SSS \rightarrow {{\rm GL}} (V)$$
of real algebraic groups.
The representation $\rho$  decomposes into irreducible representations $\rho_{p, q}$ with multiplicities $m_{p,q}$  
\begin{equation}
\rho = \bigoplus_{q\leq p}\,\,\,  m_{p, q}\,\, \rho_{p, q},
\nonumber\end{equation}
\begin{equation}\rho_{p,q}(re^{i\phi}) \,\, := \,\, r^{p+q} \left[ \begin{array}{cc}
\cos(p-q)\phi & - \sin(p-q)\phi \\
\sin(p-q)\phi & \cos(p-q) \phi 
\end{array} \right] \;\;\;\; {\rm for} \; \; p\neq q, \; p,q \in{\Z}
\nonumber\end{equation}
\begin{equation}
\rho_{p,p} (re^{i\phi}) \, := \, r^{2p} 
\left[ \begin{array}{c}
1
\end{array}\right].
\nonumber\end{equation}
Certainly, the complexification of the representation $\rho_{p,q}$ for $q < p$ decomposes into two one-dimensional 
$\C$-vector spaces:
\begin{equation} 
\rho_{p,q} \otimes _{\R} {\C} = \rho_{p,q}^{\C} \oplus \rho_{q,p}^{\C} , 
\label{rho p q otimes C decomposition}\end{equation}
where
\begin{equation} \rho_{m,n}^{\C}  (z)=z^m {\overline z}^n .
\end{equation}

\noindent
Consider the real Lie algebra representation (the derivative of $\rho$):
$${\mathcal L}(\rho)\, : \, {\C} \rightarrow {{\rm End}} (V).$$
For $q \leq p$ the representation ${\mathcal L}(\rho_{{p, q}})$ is also two-dimensional
$${\mathcal L}(\rho_{{p, q}})(1) = (p + q) I\; \; {{\rm and}} \; \;  
{\mathcal L}(\rho_{{p, q}})(i) = (p - q) J ,$$
where
$$ I= \left[\begin{array}{cc} 1 & 0 \\  
0 & 1 \end{array} \right] , \;
J = \left[ \begin{array}{cc} 0 & -1 \\  
1 & 0 \end{array} \right]  $$
\medskip

\noindent
For $p=q$ 
$${\mathcal L}(\rho_{{p,p}})(1) = 2 p\; \; {{\rm and}} \; \;  
{\mathcal L}(\rho_{{p, p}})(i) = 0 .$$
If we put $E := {\mathcal L}(\rho)(1)$ and $T := {\mathcal L}(\rho)(i)$ then we get equations 
(\ref{hp}) and (\ref{hk}).
The condition (\ref{hk}) is fulfiled because $p-q$ and $p+q$ have the same parity.
\medskip

Now let us assume that conditions (\ref{hp}) and (\ref{hk}) hold.
Observe that $\sinh (z)$ and $\sin (z)$ have single zeros in the complex plane.
Moreover (\ref{hp}) and (\ref{hk}) imply that the 
complexifications $E \otimes 1$ and  $T \otimes 1 \in {\rm End} (V \otimes_{{\R}} {\C})$ have 
common eigenbasis. From this it follows that the endomorphisms 
$E, T \in {\rm End}_{\R} (V)$ have common Jordan decomposition into eigenspaces of dimension $1$ or $2.$ 
We define a representation 
$$\rho\, :\, {\C}^{\times} \rightarrow \rm{GL} (V),$$
$$\rho (e^{x +iy}) = \exp (x E + y T) \; \;  {\rm for} \; \; x,y \in {\R}.$$
$\rho$ is an algebraic representation, because the equality (\ref{hk}) holds. 
The representation $\rho$ gives the Hodge structure on $V. $ 
\end{proof}


\section{Hodge structures via single operator}
Let 
$\sigma (z) $ be the Weierstrass' sigma function for the lattice
generated by $\omega_1= 1-i$ and $\omega_2=1+i:$ 
\begin{equation}
\sigma (z) :=z \prod _{(k_1,k_2)
\neq (0,0)}\left( 1-\frac{z}{k_1\omega_1+k_2 \omega_2} \right)
\exp \left[\frac{z}{k_1\omega _1+k_2 \omega_2}+
\frac{1}{2}\left(\frac{z}{k_1\omega_1+k_2 \omega_2}\right)^2\right]
\nonumber\end{equation} 

\noindent
\begin{theorem}\label{Hodge structure via single operator} 
For operators $E, T \in {\rm End}_{\R} (V)$ considered above
let $S := E+T$. We get the following equality
\begin{equation}
\label{hw} \sigma (S)=0.
\end{equation}
Conversely every $S \in {\rm End}_{\R} (V)$ satisfying condition (\ref{hw})
gives a unique pair $(E,T)$ of operators in 
${\rm End}_{\R} (V)$ such that $S = E + T$ and the conditions (\ref{hp}) and (\ref{hk}) hold.
\end{theorem}

\medskip

\noindent
\begin{proof} It is clear that $S = E + T$ satisfies the equation (\ref{hw}).
Conversely, assume that an operator $S \in End_{\R} (V)$ satisfies (\ref{hw}).
Since the $\sigma$ function has zeros of order 1, we observe that the complexification 
of $S$ is diagonalizable.
We get the operators $E$ and $T$ considering equation
\begin{equation} S (v) = \lambda v 
\end{equation}
in the complexification of $V$. 
The eigenvalues have integer real and imaginary parts with the same parity:
\begin{equation} \label{ogr} 
\lambda = a+ib, \quad\quad a,b \in {\bf Z}, \quad\quad a-b \in 2{\bf Z}. 
\end{equation}
Moreover we define the operators $E,T$ in such a way that their
complexifications acting on the eigenvector $v$ of $S$ have form:
$E (v) = av$ i $T(v) = ibv$ where $S (v) =(a+ib)v.$ Operators $E$ and $T$ 
satisfy equations (\ref{hp}) and (\ref{hk}). The operators $E$ and $T$ are uniquelly
determined. Indeed, if $S = E^{\prime} + T^{\prime}$ such that $E^{\prime}$ and 
$T^{\prime}$ satisfy (\ref{hp}) and (\ref{hk}) then it is clear that 
$[E^{\prime}, \, S] = 0$ and $[T^{\prime}, \, S] = 0.$  
\end{proof}
\medskip
\begin{remark}
For certain Hodge structures the set of eigenvalues of the comlexification of $S$ 
has further obstructions beyond (\ref{ogr}). In this case $S$ satisfies the equation
$g(S)=0,$ where $g (z)$ is an analytic function that divides $\sigma (z)$ in such a way that
$\frac{\sigma (z)}{g (z)}$ is also an analytic function on the whole complex plane.
\end{remark}

\begin{remark}
In our work in progress we define certain deformations of Hodge structures that arise in a natural
way in mathematical physics (see \cite{banmil}, \cite{holom1}, \cite{holom2}).
\end{remark}

\end{document}